\documentclass[letterpaper, 10 pt, conference]{ieeeconf}
\usepackage[margin=1in]{geometry}
\usepackage{amsmath}
\usepackage{amsfonts}
\usepackage{graphicx}
\graphicspath{{./images/FirstRevision/}}

\newtheorem{theorem}{Theorem}
\newtheorem{corollary}{Corollary}
\newtheorem{definition}{Definition}
\newtheorem{remark}{Remark}

\newtheorem{proposition}{Proposition}
\usepackage{caption}
\usepackage{subcaption} 
\usepackage{xcolor}
\usepackage{algpseudocode}
\usepackage{algorithm}
\usepackage{biblatex}
\usepackage{changes}
\usepackage{bbm}
\definecolor{kthblue}{rgb}{0.098, 0.329, 0.651}
\definecolor{newgreen}{RGB}{38, 115, 77}
\definecolor{newpurple}{RGB}{102, 0, 204}
\definecolor{newred}{RGB}{153, 0, 0}
\definecolor{newbrown}{RGB}{115, 77, 38}
\definecolor{neworange}{RGB}{255, 153, 0}
\usepackage{mathtools}

\title{
Symmetrizable systems}
\author{Hamed Taghavian and Jens Sjölund}
\begin{document}
\maketitle


\begin{abstract}
Transforming an asymmetric system into a symmetric system makes it possible to exploit the simplifying properties of symmetry in control problems. We define and characterize the family of symmetrizable systems, which can be transformed into symmetric systems by a linear transformation of their inputs and outputs. In the special case of complete symmetry, the set of symmetrizable systems is convex and verifiable by a semidefinite program. We show that a Khatri-Rao rank needs to be satisfied for a system to be symmetrizable and conclude that linear systems are generically neither symmetric nor symmetrizable.
\end{abstract}

\section{Introduction}
Symmetric systems, also known as reciprocal systems in the literature, are control systems that exhibit input-output symmetry in their transfer function matrices $G(s)\in\mathbb{C}^{m\times m}$ as follows
\begin{equation}\label{eqn:exsym}
    \Sigma_e G^\top(s)=G(s) \Sigma_e,
\end{equation}
where $\Sigma_e$ is a constant diagonal matrix whose diagonal elements are either $1$ or $-1$. Relation (\ref{eqn:exsym}) implies that a system is symmetric, if and only if, either its transfer function is a symmetric matrix or it becomes one by multiplying one or more of its columns by $-1$.


Symmetric systems are found in diverse application areas, such as electrical circuits, chemical reactors, mechanical systems, and power networks~\cite{Hughes,mohammadpour2010efficient,pates2022passive}. These systems have properties that simplify many control problems. For example, certain $H_{2}$ and $H_{\infty}$ optimal control problems have well-known analytical
solutions when applied to symmetric systems~\cite{pates2022passive,pates2019optimal}, the optimal linear-quadratic regulator of symmetric systems can be obtained by iterative learning control with no bias~\cite{taghavian2024optimal}, and estimated by a single trajectory of the system~\cite{drummond2024learning}.

In this paper, we extend the family of symmetric systems, by introducing symmetrizable systems. These systems may not be symmetric, but after a linear transformation of their inputs and outputs as
\begin{equation}\label{eqn:symmable}
    H(s)=K^{-1}G(s)K,
\end{equation}
they become symmetric, where $K\in\mathbb{R}^{m\times m}$ is a constant invertible matrix, called the symmetrizing gain. The simplifying properties of symmetric systems carry over to symmetrizable systems. This is realized by first transforming the asymmetric system $G(s)$ to a symmetric system $H(s)$ as (\ref{eqn:symmable}), solving a control problem for $H(s)$, and converting the result back for $G(s)$. For example, consider the static output-feedback controller $u_s(t)=K_s\, y_s(t)$ designed for the symmetric system $H(s)$ with input $u_s$ and output $y_s$. This control law is equivalent to applying the following static output feedback to $G(s)$:
\begin{equation}\label{eqn:u=KKsKy}
u(t)=KK_sK^{-1}y(t),
\end{equation}
where $u$ and $y$ are the input and output of the asymmetric system $G(s)$, respectively.

We study the symmetrizing transformation (\ref{eqn:symmable}) and provide the necessary and sufficient conditions for a given system $G(s)$ to admit such a transformation. Further, we propose a method to find a symmetrizing gain $K$ when one exists. Finally, we use the transformation~(\ref{eqn:symmable}) to obtain the analytic solution of a linear-quadratic optimal control problem for symmetrizable systems.

The preliminaries are given in Section~\ref{sec:pre}, and a background on symmetric systems is provided in Section~\ref{sec:sym}. Symmetrizable systems are introduced in Section~\ref{sec:symmable} and further discussed in Section~\ref{sec:symmability}. We give some examples to demonstrate the results in Section~\ref{sec:ex} and provide the concluding remarks in Section~\ref{sec:conc}.

\section{Notation and preliminaries}\label{sec:pre}
We denote the algebraic multiplicity of an eigenvalue $\lambda$ by $\textnormal{mul}(\lambda)$, interpret $\vert M\vert$ and $\textnormal{sgn}(M)$ element-wise, and use $\textnormal{ker}(M)$ and $\textnormal{col}(M)$ to denote the kernel and column spaces of $M$. The direct sum and the Kronecker product of matrices are denoted by $\oplus$ and $\otimes$, respectively. The (column-wise) Khatri-Rao product of two matrices $M$ and $N$ partitioned as
\begin{align*}
    M&=\begin{bmatrix}
        M_1 & M_2 & \dots & M_r
        \end{bmatrix},\quad M_{j}\in\mathbb{R}^{m\times t_{j}}\nonumber\\
        N&=\begin{bmatrix}
            N_1 & N_2 & \dots & N_r
        \end{bmatrix},\quad N_{j}\in\mathbb{R}^{n\times s_{j}}
\end{align*}
is defined by
$$
M\star N:=\begin{bmatrix}
            M_1\otimes N_1 & M_2\otimes N_2 & \dots & M_r\otimes N_r
        \end{bmatrix}.
$$
Let $n_+(S)$, $n_-(S)$, and $n_0(S)$ be the number of positive, negative, and zero eigenvalues of the symmetric matrix $S\in\mathbb{R}^{q\times q}$ respectively. Then the signature of $S$ is defined by the scalar
$$i(S):=n_+(S)-n_-(S).$$
A diagonal matrix $S\in\mathbb{R}^{q\times q}$ is called a \emph{signature} matrix if its diagonal elements are either $-1$ or $+1$. The two symmetric matrices $S_1,S_2\in\mathbb{R}^{q\times q}$ are congruent, denoted by $S_1\sim S_2$, if there is a non-singular matrix $M$ such that $S_2=M^\top S_1M$. The centralizer of the square matrix $M$, denoted by $\textnormal{Cent}(M)$, is the set of matrices that commute with $M$. The following proposition shows that the centralizer of the direct sum of two matrices $M_1$ and $M_2$ can be obtained by the direct sum of their centralizer elements.


\begin{proposition}[\cite{LMU}]\label{prop:cent}
    Let $M_1$ and $M_2$ be two square matrices with disjoint spectra (they have no common eigenvalues). Then $\textnormal{Cent}(M_1\oplus M_2)=\lbrace U_1 \oplus U_2 \vert\, U_1\in\textnormal{Cent}(M_1),U_2\in\textnormal{Cent}(M_2)\rbrace$.
\end{proposition}

\section{Symmetric systems}\label{sec:sym}
A linear system with transfer function $G(s)\in\mathbb{C}^{m\times m}$ is said to be (externally) symmetric if there is a signature matrix $\Sigma_e$ such that (\ref{eqn:exsym}) holds~\cite{willems1976sym}. The state-space
\begin{align}\label{eqn:state_space}
\begin{array}{rcl}
    \dot{x}(t)&=&Ax(t)+Bu(t) \\
    y(t)&=&Cx(t)+Du(t) 
\end{array},&\quad t\geq 0
\end{align}
is called (internally) symmetric if its associated short-hand system matrix
\begin{equation}\label{eqn:P}
P:=\begin{bmatrix}
    A & B \\
    C & D
\end{bmatrix}
\end{equation}
satisfies $\Sigma P = P^\top\Sigma$ for some signature matrix
\begin{equation}\label{eqn:sigma}
    \Sigma=\begin{bmatrix}
    -\Sigma_i & 0 \\
    0 & \Sigma_e
\end{bmatrix}.
\end{equation}
We assume the state-space (\ref{eqn:state_space}) represented by $P$ is minimal throughout the paper. In (\ref{eqn:sigma}), we call $\Sigma$ the signature matrix of the system or system inertia, $i(\Sigma)$ the system signature, $\Sigma_i$ the \emph{internal} signature matrix, and $\Sigma_e$ the \emph{external} signature matrix of the system. Symmetric systems with the external and internal signature matrices $\Sigma_e=I$ and $\Sigma_i=-I$ are called \emph{completely} symmetric~\cite{willems1976sym}. Relaxation systems are a well-known subset of these systems, in which $A\preceq 0$ and $D\succeq 0$ hold in their completely symmetric realizations.



\begin{proposition}[\cite{willems1976sym}]\label{prop:exsym=intsym}
    A linear system is symmetric if and only if it admits a symmetric minimal state-space realization.
\end{proposition}


\section{Symmetrizable systems}\label{sec:symmable}
We introduce the new family of symmetrizable systems as follows.
\begin{definition}
A linear system with transfer function $G(s)\in\mathbb{C}^{m\times m}$ is called \emph{symmetrizable} if there exists a static gain $K\in\mathbb{R}^{m\times m}$ such that the transformed system (\ref{eqn:symmable})
is symmetric.    
\end{definition}
The transformation (\ref{eqn:symmable}) induces the following equivalence relation on $m\times m$ transfer functions: $G\stackrel{\mathclap{\footnotesize\mbox{s}}}{\sim}H$ is true if there is some non-singular $K\in\mathbb{R}^{m\times m}$ such that (\ref{eqn:symmable}) holds. Hence, if a member of a class is symmetric, all members of that class are symmetrizable. If no member of a class is symmetric, then none of the members are symmetrizable. The following theorem provides the conditions that characterize all symmetrizable systems.

\begin{theorem}\label{lem:symmable}
    The system represented by (\ref{eqn:P}) is symmetrizable if and only if the following equations have a non-singular solution for $Q$: 
    \begin{align}
    PQ&=QP^\top,\,Q=Q^\top \label{condition:P}\\
    Q_{12}&=0. \label{condition:Q=0}
    \end{align}
    where $Q_{12}=[I_n\;\, 0]\,Q\,[0\;\, I_m]^\top$ is an off-diagonal block of $Q$.
\end{theorem}
\begin{proof}
According to Proposition~\ref{prop:exsym=intsym}, system $G(s)$ is symmetrizable if and only if $H(s)=K^{-1}G(s)K$ admits a symmetric minimal state-space realization for some $K\in\mathbb{R}^{m\times m}$. Assuming $P$ in (\ref{eqn:P}) is a system matrix of $G(s)$, all the minimal realizations of $H(s)$ are represented by the system matrix
\begin{align*}
    \begin{bmatrix}
        T^{-1}AT & T^{-1}BK\\
        K^{-1}CT & K^{-1}DK
    \end{bmatrix}=L^{-1}PL,
\end{align*}
where
\begin{equation}\label{eqn:L}
L:=\begin{bmatrix}
    T & 0\\
    0 & K
\end{bmatrix}
\end{equation}
is non-singular. Therefore, system $G(s)$ is symmetrizable if and only if there is some block-diagonal non-singular matrix $L$ as (\ref{eqn:L}) and some signature matrix $\Sigma$ as (\ref{eqn:sigma}) such that
\begin{equation}\label{eqn:sigmaLpL}
    \Sigma L^{-1}PL=L^\top P^\top L^{-\top}\Sigma.
\end{equation}
Define $Q:=L\Sigma L^\top$ and partition it as
$$
Q=\begin{bmatrix}
    Q_{11}& Q_{12}\\
    Q_{12}^\top& Q_{22}
\end{bmatrix},
$$
where $Q_{11}\in\mathbb{R}^{n\times n}$ and $Q_{22}\in\mathbb{R}^{m\times m}$. Then equation (\ref{eqn:sigmaLpL}) can be written as (\ref{condition:P})-(\ref{condition:Q=0}) where
\begin{align}
    Q_{11}&=-T\Sigma_i T^\top, \label{condition:Q=T}\\
    Q_{22}&=K\Sigma_e K^\top. \label{condition:Q=K}
\end{align}
If the solution $Q$ to equations (\ref{condition:P})-(\ref{condition:Q=0}) is singular, at least one of its sub-matrices $Q_{11}$ and $Q_{22}$ is singular, because $Q$ is block-diagonal by (\ref{condition:Q=0}). In this case, both $T$ and $K$ cannot be non-singular at the same time in (\ref{condition:Q=T})-(\ref{condition:Q=K}). However, both $T$ and $K$ must be non-singular for $L$ to be non-singular in (\ref{eqn:L}).
In contrast, if the solution $Q$ to equations (\ref{condition:P})-(\ref{condition:Q=0}) is non-singular, one can always find non-singular $T$ and $K$ satisfying (\ref{condition:Q=T})-(\ref{condition:Q=K}). To see this, note that since both $Q_{11}$ and $Q_{22}$ are non-singular in this case, there are signature matrices $\Sigma_i$, $\Sigma_e$ such that
\begin{equation}\label{eqn:Q_sim_Sigma}
Q_{11}\sim -\Sigma_i \textnormal{ and } Q_{22}\sim \Sigma_e.
\end{equation}
By Sylvester's law of inertia, equations (\ref{condition:Q=T}) and (\ref{condition:Q=K}) have non-singular solutions for $T$ and $K$ if and only if $\Sigma_i$, $\Sigma_e$ are chosen as (\ref{eqn:Q_sim_Sigma}). Therefore, system $G(s)$ is symmetrizable if and only if (\ref{condition:P})-(\ref{condition:Q=0}) have a non-singular solution for $Q$.
\end{proof}


\begin{remark}
    Symmetrizable systems include symmetric systems as a subset. This follows from choosing $K=I$ in (\ref{eqn:symmable}).
\end{remark}

\begin{remark}
    Varying $L$ (\ref{eqn:L}) does not change the class under the equivalence relation $\stackrel{\mathclap{\footnotesize\mbox{s}}}{\sim}$. Therefore, symmetrizability is independent of realization.
\end{remark}

When the system represented by the system matrix (\ref{eqn:P}) is symmetrizable, one can find a symmetrizing gain $K$ and a symmetric realization $T$ based on the solution $Q$ of the equations (\ref{condition:P})-(\ref{condition:Q=0}). Such matrices $T$ and $K$ are not unique, with one instance being
\begin{align}\label{eqn:1sol}
T=F_1 \vert D_1\vert^{1/2}, \;
K=F_2 \vert D_2\vert^{1/2}  
\end{align}
where $D_1$ ($D_2$) is the diagonal matrix of eigenvalues of $Q_{11}$ ($Q_{22}$) associated with an orthonormal matrix of eigenvectors $F_1$ ($F_2$) sorted such that $\textnormal{sgn}(D_1)=-\Sigma_i$ ($\textnormal{sgn}(D_2)=\Sigma_e$).

The internal and external signature matrices $\Sigma_i$ and $\Sigma_e$ of the symmetrized system are determined by the signatures of the submatrices $Q_{11}$ and $Q_{22}$, respectively, via (\ref{eqn:Q_sim_Sigma}).
In particular, when we are only interested in complete symmetry ($\Sigma_i=-I$ and $\Sigma_e=I$), we can enforce $Q\succ 0$ in (\ref{condition:P}), which makes the symmetrizability conditions convex.
Therefore, system (\ref{eqn:state_space}) can be symmetrized into a completely symmetric system if and only if the following semidefinite program is feasible
$$
    \begin{array}[c]{rl}
    \text{find} & Q\in \mathbb{R}^{(n+m)\times(n+m)}\\
    \mbox{subject to}
    & Q\succ 0\\
    & PQ=QP^\top \\
    & [I_n\;\, 0]\,Q\,[0\;\, I_m]^\top=0.
    \end{array}
$$
However, in the general case when $\Sigma_i$ and $\Sigma_e$ are not fixed, the condition on non-singularity of $Q$ in Theorem~\ref{lem:symmable} is non-convex. Even specifying the system inertia $\Sigma$ (and hence the sign of the eigenvalues of $Q$) a priori does not resolve this issue; except for the positive and negative definite cones, all the other cones in which $Q$ has fixed eigenvalue signs are non-convex. Therefore, it can be difficult to verify the conditions in Theorem~\ref{lem:symmable} numerically in the general case. Alternative conditions are given in the next section.


\section{Symmetrizability}\label{sec:symmability}

In this section, we obtain conditions that are more computationally amenable than Theorem~\ref{lem:symmable}. These conditions are either necessary (Theorem~\ref{thm:KR}) or both necessary and sufficient (Corollary~\ref{cor:realdistinct}) for a linear system to be symmetrizable. We also determine the possible inertia of the symmetrized system based on a minimal realization.

First, we use the symmetric decomposition to show that equation (\ref{condition:P}) always has a non-singular solution. According to this decomposition, for every square matrix $P$ there are two real symmetric matrices $Q$ and $N$ such that~\cite{taussky}\looseness=-1
\begin{equation}\label{eqn:P=QN}
P=QN,    
\end{equation}
where $Q$ is non-singular. Since $N=Q^{-1}P$ is symmetric, one has $Q^{-1}P=P^\top Q^{-1}$, where pre- and post-multiplying $Q$ recovers (\ref{condition:P}). Therefore, $Q$ is a non-singular solution of (\ref{condition:P}) if and only if it is the left non-singular factor in a symmetric decomposition of $P$ (which always exists). All these factors can be found analytically using the following proposition.



\begin{proposition}[\cite{uh1974}]\label{prop:S}
    Let $J=V^{-1}PV=\bigoplus_{i=1}^{b}J_i$ be the real Jordan normal form of $P$ and let $E_i$ be a matrix of the same size as the Jordan block $J_i$ and of the form
    \begin{equation}\label{eqn:matrix_Ei}
    E_i=\begin{bmatrix}
        0& & 1\\
         &.^{.^.} &\\
         1 & & 0
    \end{bmatrix}.
    \end{equation}
    All possible left factors in symmetric decompositions (\ref{eqn:P=QN}) of $P$ are given by
\begin{equation}\label{eqn:Q}        Q=VU\bigl(\bigoplus_{i=1}^{b}\sigma_iE_i\bigr) U^\top V^\top,
\end{equation}
    where $\sigma_i\in\lbrace -1,+1\rbrace$ and $U\in \textnormal{Cent}(J)$ is non-singular.
\end{proposition}


In addition to (\ref{condition:P}), condition (\ref{condition:Q=0}) also needs to be satisfied for a system to be symmetrizable, according to Theorem~\ref{lem:symmable}. Therefore, a system is symmetrizable if and only if its system matrix (\ref{eqn:P}) has a \emph{block-diagonal} left factor in a symmetric decomposition, \emph{i.e.}, a factor (\ref{eqn:Q}) that satisfies (\ref{condition:Q=0}).

Based on this result, we provide a simplified condition in Theorem~\ref{thm:KR}, which requires the following definitions for a clear exposition. We let $P$ have $r:=p+q$ distinct real eigenvalues $\lambda_1,\,\lambda_2,\,\dots,\lambda_p,$ and distinct pairs of complex-conjugate eigenvalues $$(\lambda_{p+1},\bar{\lambda}_{p+1}),\,(\lambda_{p+2},\bar{\lambda}_{p+2}),\,\dots,\,(\lambda_{p+q},\bar{\lambda}_{p+q}),$$
regardless of multiplicities. We define
$$
t_j:=\left\lbrace
\begin{array}{ll}
   \textnormal{mul}(\lambda_j), & \lambda_j\in\mathbb{R}\\
   2\,\textnormal{mul}(\lambda_j), & \lambda_j\not\in\mathbb{R}
\end{array}
\right., \quad j=1,2,\dots,r,
$$
and rearrange the Jordan blocks in the real Jordan normal form of $P$ as
$J=V^{-1}PV:=\bigoplus_{j=1}^{r}\hat{J}_j$,
where $\hat{J}_j$ contains all the Jordan blocks with eigenvalue $\lambda_j$ if $\lambda_j\in\mathbb{R}$, and all the Jordan blocks with eigenvalues $\lambda_j$ and $\bar{\lambda}_j$ otherwise. Therefore, the blocks with the same real eigenvalues are next to each other and the blocks with the same pairs of complex-conjugate eigenvalues are also adjacent. We also partition $V$ as
\begin{equation}\label{eqn:V}
    V=\begin{bmatrix}
    W_{n\times (n+m)}\\Z_{m\times (n+m)}
\end{bmatrix},
\end{equation} 
where $W, Z$ are further partitioned column-wise as follows
    \begin{align}\label{eqn:W,Z}
        W&=\begin{bmatrix}
            W_1 & W_2 & \dots & W_{r}
        \end{bmatrix},\quad W_j\in\mathbb{R}^{n\times t_j}\nonumber\\
        Z&=\begin{bmatrix}
            Z_1 & Z_2 & \dots & Z_{r}
        \end{bmatrix},\quad Z_j\in\mathbb{R}^{m\times t_j}.
    \end{align}
\begin{theorem}\label{thm:KR}
     If the system represented by (\ref{eqn:P}) is symmetrizable, then
\begin{equation}\label{eqn:Khatricond}
    \textnormal{Rank}(Z\star W)< t_1^2+\dots+t_r^2.
\end{equation}
\end{theorem}
\begin{proof}
A system represented by the system matrix (\ref{eqn:P}) is symmetrizable if and only if there is a factor (\ref{eqn:Q}) that satisfies (\ref{condition:Q=0}). Let us write this factor as
\begin{equation}\label{eqn:Q=VXV}
Q=VXV^\top,    
\end{equation}
where
\begin{equation}\label{eqn:X}
    X:=U\bigl(\bigoplus_{i=1}^{b}\sigma_iE_i\bigr)U^\top.
\end{equation}
Then equation (\ref{condition:Q=0}) may be written as $[I_n\;\, 0]\,Q\,[0\;\, I_m]^\top=WXZ^\top=0$ which, after vectorizing, takes the form
\begin{equation}\label{eqn:ZkronW}
\left(Z\otimes W \right)\textnormal{vec}(X)=0.
\end{equation}
Next, we simplify the expression (\ref{eqn:X}) for $X$ and plug it in (\ref{eqn:ZkronW}). Note that the permutation of Jordan blocks does not generate new factors $Q$ in (\ref{eqn:Q}). To see this, let $Y=\prod_k \mathcal{E}_k$ be an arbitrary product of elementary block permutation matrices $\mathcal{E}_k$. Then substituting $V$, $U$, $\bigoplus_{i=1}^{b}\sigma_iE_i$ in (\ref{eqn:Q}) by 
$$
V Y, \; Y^{-1}UY, \;
Y^{-1}\bigl(\bigoplus_{i=1}^{b}\sigma_iE_i\bigr)Y
$$
respectively, does not change the factor $Q$. Therefore, without loss of generality, we assume that the Jordan blocks with the same eigenvalues are placed next to each other in $J$ as follows
\begin{equation}\label{eqn:J=Jhat}
    J=\bigoplus_{i=1}^{b}J_i=\bigoplus_{j=1}^{r}\hat{J}_j.
\end{equation}
Since $\hat{J}_j$ have disjoint spectra, one can invoke Proposition~\ref{prop:cent} to write
$$
\textnormal{Cent}(J)=\left\lbrace \bigoplus_{j=1}^{r}\hat{U}_j\,\biggl\vert\biggr.\,
\hat{U}_j\in\textnormal{Cent}\bigl(\hat{J}_j\bigr),\; j=1,2,\dots,r
\right\rbrace.
$$
In a similar way to (\ref{eqn:J=Jhat}), one can regroup the matrices $\sigma_i E_i$ as
$\bigoplus_{i=1}^{b}\sigma_iE_i=\bigoplus_{j=1}^{r}\hat{E}_j$ where $\hat{E}_j$ has the same size as $\hat{J}_j$. Then expression (\ref{eqn:X}) can be written in the following form 
\begin{align}\label{eqn:X=sum}
X&=\bigoplus_{j=1}^{r}\hat{U}_j\bigl(\bigoplus_{j=1}^{r}\hat{E}_j\bigr)
\bigoplus_{j=1}^{r}\hat{U}^\top_j \nonumber\\
&=\bigoplus_{j=1}^{r}\hat{U}_j\hat{E}_j\hat{U}^\top_j\nonumber\\
&:=\bigoplus_{j=1}^{r} X_j.
\end{align}
By plugging in (\ref{eqn:X=sum}), equation (\ref{eqn:ZkronW}) takes the form
\begin{equation}\label{eqn:ZWx=0}
    \left(Z\star W \right)x=0,
\end{equation}
where
\begin{align}\label{eqn:ZstarW}
    \left(Z\star W \right)&=\begin{bmatrix} Z_1\otimes W_1 & \dots & Z_r\otimes W_r\end{bmatrix}\nonumber\\
    x&=[\textnormal{vec}(X_1)^\top,\dots,\textnormal{vec}(X_r)^\top]^\top.
\end{align}
However according to (\ref{eqn:Q=VXV}), $Q$ is non-singular if and only if $X$ is non-singular. Therefore, if equations (\ref{condition:P}), (\ref{condition:Q=0}) have a non-singular solution for $Q$, then the Khatri-Rao product (\ref{eqn:ZstarW}) has a positive nullity. This condition is equivalent to (\ref{eqn:Khatricond}) by the rank-nullity theorem.
\end{proof}

\begin{remark}
The Khatri-Rao product of two matrices is \emph{generically} full rank~\cite{KhatriRank}. Hence according to (\ref{eqn:Khatricond}), random linear systems (with system eigenvectors drawn from continuous probability density functions) are neither symmetric nor symmetrizable with probability one.
\end{remark}


The kernel space of the Khatri-Rao product $Z\star W$ encodes all the non-singular solutions $Q$ of the symmetrizability equations (\ref{condition:P})-(\ref{condition:Q=0}) (cf. equation (\ref{eqn:ZWx=0})). The next corollary gives more insight in the special case where $P$ has $n+m$ real distinct eigenvalues.
\begin{corollary}\label{cor:realdistinct}
     Assume $P$ has real distinct eigenvalues. The system represented by $P$ is symmetrizable if and only if there is some $x\in \textnormal{ker}(Z\star W )$ with non-zero elements. Furthermore, all the possible signatures of the symmetrized system are given by
    $$   i(\Sigma)=\sum_{j=1}^{n+m}\textnormal{sgn}(x_j).
    $$
\end{corollary}
\begin{proof}
    When $P$ has real distinct eigenvalues the matrices $E_i$ (\ref{eqn:matrix_Ei}) equal unity, $J$ is diagonal and $\textnormal{Cent}(J)$ is the set of all matrices that are simultaneously diagonalizable with $J$, that is all diagonal matrices. Therefore, all the solutions of (\ref{condition:P})-(\ref{condition:Q=0}) can be written as (\ref{eqn:Q=VXV}) where $X$ is a non-singular diagonal matrix, whose diagonal elements are given by $x$ in (\ref{eqn:ZstarW}).
\end{proof} 


When $P$ has distinct real eigenvalues, the matrix $Q$ is singular if and only if $x_j=0$ holds for some $j=1,2,\dots,m+n$, according to (\ref{eqn:Q=VXV}). In this case, Corollary~\ref{cor:realdistinct} makes it easy to isolate the singularities that make the conditions of Theorem~\ref{lem:symmable} non-convex and verify symmetrizability using linear feasibility programs. In particular, the system represented by $P$ is symmetrizable with signature $i$, if and only if the following linear program is feasible for one $e\in\lbrace -1,+1\rbrace^{n+m}$ where $\sum_{j=1}^{n+m}e_j=i$:
$$
    \begin{array}[c]{rl}
    \text{find} & x\in \mathbb{R}^{n+m}\\
    \mbox{subject to}
    & (Z\star W)\, x=0 \\
    & e_j x_j>0,\quad j=1,2,\dots,m+n.
    \end{array}
$$

\section{Examples}\label{sec:ex}
In this section, we show how symmetrizability extends symmetry in a physical system (Example~\ref{ex:tank}), check whether a given system is symmetrizable, determine the possible system signatures, and compute symmetrizing gains (Example~\ref{ex:numerical}) and extend the analytic solution of an optimal control problem for symmetric systems to symmetrizable systems (Example~\ref{ex:optimal}).
\subsection{A classical multi-tank system}\label{ex:tank}
Consider a quadruple-tank process with two inputs (voltages applied to the pumps) and two outputs (levels of the lower tanks)~\cite{nunes1998modeling}. This process can be described by the transfer function
\begin{equation}\label{eqn:ex:G}
G(s)=\begin{bmatrix}\frac{c_{11}}{1+sT_1} &\frac{c_{12}}{(1+sT_1)(1+sT_3)}\\
\frac{c_{21}}{(1+sT_2)(1+sT_4)} & \frac{c_{22}}{1+sT_2}\end{bmatrix},
\end{equation}
where
\begin{align}\label{eqn:ex:par}
    c_{11}&=\gamma_1 k_1 T_1k_c/A_1, \nonumber\\
    c_{12}&=(1-\gamma_2)k_2T_1k_c/A_1\neq 0, \nonumber\\
    c_{21}&=(1-\gamma_1)k_1T_2k_c/A_2\neq 0, \nonumber\\
    c_{22}&=\gamma_2 k_2T_2k_c/A_2.
\end{align}
In (\ref{eqn:ex:G})-(\ref{eqn:ex:par}), $A_i$ is the cross-section of the $i$th tank, $T_i$ is the time constant associated with the $i$th tank, constants $k_i>0$ and $\gamma_i\in[0,1]$ are determined by the $i$th pump and valve settings and $k_c>0$. From definition (\ref{eqn:exsym}), this system is symmetric if and only if
\begin{align}
    T_1&=T_2,T_3=T_4 \;\textnormal{or}\; T_1=T_4,T_2=T_3 \label{ex:cond1}\\
    c_{12}&=\pm \,c_{21}. \label{ex:cond2}
\end{align}
However, the first condition (\ref{ex:cond1}) is sufficient for the system to be symmetrizable. Hence symmetrizability imposes a less restrictive condition on the tank parameters than symmetry.

\subsection{A numerical example}\label{ex:numerical}
Consider a system described by the minimal state-space equations (\ref{eqn:state_space}) where
\begin{align*}
    A&=\begin{bmatrix}
    3.6000 &  -0.6333\\
   -0.3000 &   3.4000\end{bmatrix},\\   
   B&=\begin{bmatrix}
-2.1400 &  1.3200 &   3.6400\\
-1.8600 &  2.2800 &   9.9600
\end{bmatrix},\\
C&=\begin{bmatrix}
    0.0500  & -0.3833\\
    2.2250  & -1.9083 \\
   -0.5000  &  0.4667
\end{bmatrix},\\
D&=\begin{bmatrix}
    2.2000  & -1.8000  & -8.4000 \\
   -2.4000  & -4.4000 & -37.2000\\
    0.4000  &  1.4000 &  10.2000
\end{bmatrix}.
\end{align*}
It can be verified that this system is not symmetric, by calculating its transfer function and checking the condition (\ref{eqn:exsym}). To see whether it is symmetrizable, we calculate the eigenvectors (\ref{eqn:V}) of the system matrix $P$ (\ref{eqn:P}) and obtain
$$
\textnormal{ker}(Z\star W )=\textnormal{col}\Biggl(\begin{bmatrix}
1&0&0&0&0\\
0&1&1&1&1
\end{bmatrix}^\top\Biggr).
$$
Matrix $P$ has $n+m=5$ real distinct eigenvalues. Therefore, according to Corollary~\ref{cor:realdistinct}, this system is symmetrizable with the signatures $i(\Sigma)\in\lbrace \pm 5,\pm 3\rbrace$. Next, we demonstrate how the symmetrizing gain $K$ is computed to realize a desired signature. Let us choose $i(\Sigma)=-3$ and take a point of the kernel space $\textnormal{ker}(Z\star W )$ with the corresponding sign, \emph{e.g.},
$$
x=\begin{bmatrix}1& -1& -1& -1& -1\end{bmatrix}^\top.
$$
Then $Q$ is calculated as (\ref{eqn:Q=VXV}) and the gain $K$ is obtained by solving (\ref{condition:Q=K}), \emph{e.g.} through (\ref{eqn:1sol}) as
$$
K=\begin{bmatrix}
2.3460 & -0.0576 & -2.7399 \\
-3.0744 &  0.0637 & -2.1109 \\
0.9380 & 0.3529 & -0.0658
\end{bmatrix}.
$$
One can verify that $H(s)=K^{-1}G(s)K$ is symmetric with $i(\Sigma_e)=1$.

\subsection{An optimal control problem}\label{ex:optimal}
Consider the linear system described by the minimal state-space equations
\begin{align}\label{eqn:state_space_ex}
\begin{array}{rcl}
    \dot{x}(t)&=&Ax(t)+Bu(t)+w(t) \\
    y(t)&=&Cx(t)+Du(t) 
\end{array},&\quad t\geq 0
\end{align}
where $x(0)=0$ and $w(t)\in\mathbb{R}^n$ is the disturbance input. We would like to find an output-feedback linear controller that stabilizes the system and minimizes the performance measure
$$
\mathcal{J}(R,S)=\sup_{w\in\mathcal{W}(S)}\int_{0}^{\infty} y(t)^\top R y(t)+\alpha^2 u(t)^\top R u(t) dt
$$
where $\alpha>0$ balances the control effort versus output regulation in the objective function and
$$
\mathcal{W}(S)=\left\lbrace w \,\Big\vert\, \int_{0}^{\infty} w^\top(t)Sw(t)dt\leq 1\right\rbrace.
$$
The constants $R,S\succ 0$ are determined based on the dynamics (\ref{eqn:state_space_ex}) in a way explained shortly. This problem was introduced and solved for relaxation systems in \cite{pates2019optimal}, for which $R=I$. A system is of relaxation type, if and only if it admits a completely symmetric realization (one with $\Sigma=I$) and the eigenvalues of $A$ and $-D$ (in any of its minimal realizations) are on the closed left-half plane. We extend this result to \emph{symmetrizable} systems with $\Sigma=I$ in which the eigenvalues of $A$ and $-D$ are located on the closed left-half plane. These systems include relaxation systems as a proper subset and allow for $R\succ 0$ as any positive definite matrix.

Assume that system (\ref{eqn:state_space_ex}) is symmetrizable with $\Sigma=I$. Then the equations (\ref{condition:P})-(\ref{condition:Q=0}) have a solution $Q\succ 0$ which can be used in (\ref{condition:Q=T})-(\ref{condition:Q=K}) to find a similarity transform matrix $T$ and a symmetrizing gain $K$ as follows
\begin{align*}
    T=Q_{11}^{1/2},\;
    K=Q_{22}^{1/2}.
\end{align*}
If the eigenvalues of $A$ and $-D$ are located on the closed left-half plane, equations (\ref{condition:P})-(\ref{condition:Q=0}) can be written as
\begin{align*}
    &A^\top T^{-2}=T^{-2} A\preceq 0,\;
    (K^{-1}C)^\top= T^{-2} BK,\\
    &K^{-1}DK\succeq 0,
\end{align*}
which certifies that the symmetrized system (\ref{eqn:symmable}) with the following state space is of relaxation type
\begin{align}\label{eqn:state_space_s}
    \dot{\hat{x}}(t)&=T^{-1}AT\hat{x}(t)+T^{-1}BK u_s(t)+w_s(t)\nonumber\\
    y_s(t)&=K^{-1}CT\hat{x}(t)+K^{-1}DK u_s(t),
\end{align}
where $y_s(t)=K^{-1}y(t)$, $u_s(t)=K^{-1}u(t)$, and $w_s=T^{-1}w(t)$. Choosing 
\begin{align*}
 R=K^{-2},\;
 S=T^{-2}
\end{align*}
sets the optimal control problem in the recognized form of \cite{pates2019optimal}, in terms of the relaxation system (\ref{eqn:state_space_s}), its inputs $u_s$, $w_s$ and its output $y_s$. Based on the results of \cite{pates2019optimal}, the performance measure $\mathcal{J}(K^{-2},T^{-2})$ is minimized under the relaxation dynamics (\ref{eqn:state_space_s}) with a static output-feedback controller with the closed-form expression
$$
u_s(t)=
-\alpha^{-1}(K^{-1}DK-K^{-1}CA^{-1} BK)
y_s(t).
$$ 
Therefore, the optimal feedback controller from $y$ to $u$ is given by
$$
u(t)=
-\alpha^{-1}(D-CA^{-1}B)y(t).
$$

\section{Conclusion}\label{sec:conc}
We studied the problem of transforming a linear system into a symmetric system using a static gain. We have derived the conditions for this transformation to exist and provided a method to obtain both the symmetrizing gains and symmetric realizations with different signatures.

Symmetrizability is a weaker condition than symmetry. Yet, it allows one to apply the simplifying properties of symmetry when controlling symmetrizable systems. To demonstrate this point, we revisited an optimal control problem and extended its analytic solution for symmetric systems to symmetrizable systems.

Finally, we have shown that linear systems are generically not symmetrizable. Therefore, a future research direction would be to approximate systems by the ``nearest'' symmetrizable systems.

\printbibliography
\end{document}